\documentclass[a4paper,12pt,twoside]{article}

\title{Small eigenvalues of the Hodge-Laplacian 
    with sectional curvature bounded below}

\author{\Large Colette Ann\'e and Junya Takahashi}

\date{1 July 2025}

\usepackage{latexsym}
\usepackage{amsmath}
\usepackage{amsthm}
\usepackage{amssymb}
\usepackage{amscd}
\usepackage{graphicx} 
\usepackage{color} 

\usepackage[french,english]{babel}

\usepackage{float}

\usepackage{tikz}
\usepackage{tikz-cd}
\usetikzlibrary{patterns}
\usetikzlibrary{arrows}
\usetikzlibrary{positioning}

\usepackage{hyperref}
\usepackage{url}

\usepackage{mathrsfs}

\usepackage{fancyhdr}
\newtheorem{thm}{Theorem}[section]

\newtheorem{prob}[thm]{Problem}

\newtheorem{lem}[thm]{Lemma}

\newtheorem{rem}[thm]{Remark}

\newtheorem{conj}[thm]{Conjecture}

\renewcommand{\labelenumi}{$(\arabic{enumi})$} 

\numberwithin{equation}{section}

\newcommand{\dint}{\displaystyle \int}

\newcommand{\diam}{\operatorname{diam}}

\newcommand{\Ric}{\operatorname{Ric}}

\newcommand{\supp}{\operatorname{supp}}

\newcommand{\Tub}{\operatorname{Tub}}
\newcommand{\vol}{\operatorname{vol}}

\def\a{\mathop{\mathrm{\alpha}}\nolimits}

\def\e{\mathrm{\varepsilon}}

\def\vphi{\mathop{\mathrm{\varphi}}\nolimits}

\def\N{\mathop{\mathrm{{\Bbb N}}}\nolimits}

\def\R{\mathop{\mathrm{{\Bbb R}}}\nolimits}

\begin{document}

\maketitle

\vspace{-0.2cm}

\begin{center}
 Dedicated to Professor Bruno Colbois on the occasion of his 65th birthday
\end{center}

\vspace{-0.2cm}

\begin{abstract}
 For each degree $p$ and each natural number $k \geq 1$, 
 we construct a one-parameter family of Riemannian metrics on any oriented 
 closed manifold with volume one and the sectional curvature bounded below 
 such that the $k$-th positive eigenvalue of the Hodge-Laplacian acting 
 on differential $p$-forms converges to zero. 
 This result imposes a constraint on the sectional curvature for our previous 
 result in \cite{small-vol[24]}.
\end{abstract}

\footnotetext{$2020$ \it{ Mathematics Subject Classification.} 
 Primary $58J50$; Secondary $35P15,$ $53C21$, $58C40$.
 }



 \section{Introduction}

 We study the eigenvalue problems of the Hodge-Laplacian 
 $\Delta = d \delta + \delta d$ acting on $p$-forms 
 on a connected oriented closed Riemannian manifold $(M^m,g)$ of dimension $m \ge 2$.
 The spectrum of the Hodge-Laplacian consists only of non-negative eigenvalues 
 with finite multiplicity.
 We denote its {\bf positive} eigenvalues counted with multiplicity by 
\begin{equation*}
\begin{split}
  \underbrace{0 = \cdots = 0}_{b_p(M)} < {\lambda}^{(p)}_1(M,g) 
    \le {\lambda}^{(p)}_2(M,g) \le \cdots  \le {\lambda}^{(p)}_k(M,g)
    \le \cdots, 
\end{split}
\end{equation*}
 where the multiplicity of the eigenvalue $0$ is equal to the $p$-th Betti number
 $b_p(M)$ of $M$, by the Hodge-Kodaira-de Rham theory.
 In particular, it is independent of a choice of Riemannian metrics.

\vspace{0.2cm}
 In our previous paper \cite[Theorem $1.2$]{small-vol[24]}, 
 for any fixed degree $p$ with $1 \le p \le m-1$,
 we constructed a one-parameter family of Riemannian metrics 
 $\{ \overline{g}_{p,L} \}_{L>1}$ on a connected oriented closed $m$-dimensional 
 manifold $M$ with volume one such that for any natural number $k \ge 1$ 
\begin{equation} \label{eq:thm-small-vol[24]}
\begin{split}
  \lambda^{(p)}_k (M, \overline{g}_{p,L}) &\longrightarrow 0
   \  \text{ as } \  L \longrightarrow \infty.
\end{split}
\end{equation}
 If $M$ is the $m$-dimensional standard sphere ${\mathbb S}^m$, 
 then we can choose such a family of Riemannian metrics to have 
 non-negative sectional curvature (\cite[Theorem $1.1$]{small-vol[24]}). 
 These metrics are also positive Ricci curvature for $m \ge 4$.
 But, for $m=3$ and $p=1$, they are flat on some domain.

 For a general closed manifold $M$, however, the same result cannot hold any longer.
 In fact, there exist some topological obstructions to admit a Riemannian metric 
 on $M$ with non-negative Ricci curvature.
 One of the most famous obstructions is the Bochner theorem: 
 If a closed manifold $M$ admits a Riemannian metric with non-negative Ricci curvature,
 then the first Betti number must hold $b_1(M) \le b_1(T^m) = m.$

 Because of such a topological obstruction, we weaken a curvature constraint of 
 a general closed manifold $M$ from non-negative sectional curvature to the sectional 
 curvature bounded below by a negative constant. 

 In the present paper, for any closed manifold $M$ of dimension $m \ge 2$,
 we construct such a family of Riemannian metrics with the sectional curvature 
 uniformly bounded below.

\begin{thm} \label{thm:main}
 Let $M^m$ be a connected oriented closed manifold of dimension $m \ge 2.$
 For a given degree $p$ with $0 \le p \le m$, a natural number $k$
 and any $\e >0$, there exists a one-parameter family of Riemannian metrics 
 $\{ \overline{g}_{\e,p,k} \}_{\e >0}$ on $M$ with volume one and 
 the sectional curvature uniformly bounded below 
  $K_{\overline{g}_{\e,p,k}} \ge - \kappa$ 
 for some constant $\kappa >0$ such that 
\begin{equation*} 
\begin{split}
  \lambda^{(p)}_k(M, \overline{g}_{\e,p,k}) &\longrightarrow  0 
    \quad  \text{ as } \e \longrightarrow 0.
\end{split}
\end{equation*}
\end{thm}

 The construction of this one-parameter family of Riemannian metrics 
 is as follows: 
 We take an embedded $p$-dimensional sphere ${\Bbb S}^p$ into $M$ 
 whose normal bundle is trivial. 
 Then, in a tubular neighborhood of ${\Bbb S}^p$, we change 
 a disk of the normal direction to get longer and thinner, 
 while keeping its sectional curvature uniformly bounded below.

\begin{rem} \label{rem:main-thm}
\begin{enumerate}
  \renewcommand{\labelenumi}{$(\roman{enumi})$}
 \item  The Riemannian metrics $\overline{g}_{\e,p,k}$ in Theorem $\ref{thm:main}$
  depend on the degree $p$ and the number $k$ of the positive eigenvalues. \\[-0.8cm]

 \item  For the Riemannian metrics $\overline{g}_{\e,p,k}$ on $M$ in 
  Theorem $\ref{thm:main}$, from the proof, we find that the diameter 
  $\diam (M, \overline{g}_{\e,p,k}) \longrightarrow \infty$ 
  as $\e \longrightarrow 0.$  \\[-0.8cm]

 \item  For the rough Laplacian $\overline{\Delta} = \nabla^* \nabla$ acting on 
 $p$-forms and tensor fields of any type, the same statement also holds
 (See Remark \ref{rem:small-ev} $(ii)$).
\end{enumerate}
\end{rem}

 The Riemannian metrics $\overline{g}_{\e,p,k}$ in Theorem $\ref{thm:main}$
 depend also on the degree $p$ of differential forms. 
 However, by taking $m-1$ distinct embedded spheres
 ${\Bbb S}^0, {\Bbb S}^1, {\Bbb S}^2, \dots$, $ {\Bbb S}^{m-2}$ in $M$
 (see Lemma \ref{lem:k-ev-estimate}) 
 and applying the same construction in Theorem $\ref{thm:main}$ to each sphere,
 we can obtain a family of Riemannian metrics $\overline{g}_{\e,k}$ on $M$, 
 which are independent of all the degrees $p=0,1,\dots,m$,
 with small eigenvalues for all the degrees $p =0,1,2,\dots,m.$

\begin{thm} \label{thm:uniform-main}
 Let $M^m$ be a connected oriented closed manifold of dimension $m \ge 2.$
 For any $\e >0$ and a natural number $k$, 
 there exists a one-parameter family of Riemannian metrics 
 $\{ \overline{g}_{\e,k} \}_{\e >0}$ on $M$ with volume one and 
 the sectional curvature uniformly bounded below 
 $K_{\overline{g}_{\e,k}} \ge - \kappa$ for some constant $\kappa >0$ 
 such that for any degree $p$ with $0 \le p \le m$ 
\begin{equation*} 
\begin{split}
  \lambda^{(p)}_k(M, \overline{g}_{\e,k}) &\longrightarrow  0   
    \quad  \text{ as } \e \longrightarrow 0.
\end{split}
\end{equation*}
\end{thm}

\begin{rem}
 As a consequence of Theorem $\ref{thm:uniform-main}$, 
 we find that there exists no positive lower bound for the positive eigenvalue 
 of the Hodge-Laplacian on $p$-forms for any degree $p$ with $1 \leq p \leq m-1$ 
 depending only on the dimension, the volume and a lower bound of 
 the sectional curvature. 
\end{rem}

 From Remark \ref{rem:main-thm} $(ii)$, it is a natural question to ask the case 
 where the diameter is bounded in addition.
 In this case, it would be expected to exist a positive lower bound for 
 the positive eigenvalues of the Hodge-Laplacian for all the degree $p=0,1, \dots,m.$
 This was conjectured by J.\ Lott \cite[p.918]{Lott[04]-quotient}
 (See Conjecture \ref{conj:Lott}).

\vspace{0.3cm}

 The present paper is organized as follows:
 In Section $2$, we fix notations and recall basic properties of 
 the Hodge-Laplacian.
 In Section $3$, we consider the hyperbolic dumbbell and a connected sum 
 of its $k$ copies.
 In Section $4$, we construct a family of Riemannian metrics on 
 any closed manifold $M$, 
 and in Section $5$, we prove that such Riemannian manifolds have 
 small eigenvalues, which completes the proof of Theorem \ref{thm:main}.
 In Section $6$, we discuss some remarks and further studies.

 \vspace{0.5cm}

\noindent
 {\bf Acknowledgement.}
 The authors would like to thank the referees for helpful comments. 
 The second named author was partially supported by the Grants-in-Aid 
 for Scientific Research (C), Japan Society for the Promotion of Science,
 No.\ 16K05117.


\section{Notations and basic facts}

 We fix the notations used in the present paper.
 We use the same notations as in \cite{small-vol[24]}.
 Let $(M^m,g)$ be a connected oriented closed Riemannian manifold of 
 dimension $m \ge 2.$
 The metric $g$ defines the volume element $d \mu_g$ and the scalar product 
 on the fibers of any tensor bundle. The $L^2$-inner product on the space 
 of all smooth $p$-forms $\Omega^p(M)$ is defined as, 
 for any $p$-forms $\vphi, \psi$ on $M$
\begin{equation*} \label{eq:L^2-metric}
\begin{split}
  ( \vphi, \psi )_{L^2(M,g)} :&= \dint_M \langle \vphi, \psi \rangle d \mu_g
   \quad  \text{ and }  \quad
  \| \vphi \|^2 _{L^2(M,g)} := ( \vphi, \vphi )_{L^2(M,g)}.
\end{split}
\end{equation*}
 The space of $L^2$ $p$-forms $L^2 (\Lambda^p M,g)$ is the completion 
 of $\Omega^p (M)$ with respect to this $L^2$-norm.

\vspace{0.2cm}

 We now recall the basic properties used in the present paper: 

\begin{lem}\label{lem:basic}
$(1)$ The Hodge duality: 
  For all $p=0,1,\dots,m$ and any $k \ge 1$, since $\Delta * = * \Delta$, 
 we have  
\begin{equation*}
\begin{split}
 \lambda^{(m-p)}_k(M,g) &= \lambda^{(p)}_k(M,g).
\end{split}
\end{equation*}
\noindent
$(2)$ The scaling change of metrics: 
 For a positive constant $a >0$ and for all $p=0,1,\dots,m$ and any $k \ge 1$, 
 we have 
\begin{equation*}
\begin{split}
  \lambda^{(p)}_k(M,ag) &= a^{-1} \, \lambda^{(p)}_k(M,g).
\end{split}
\end{equation*}
\noindent
$(3)$ The normalization of the volume: 
 If we set the new Riemannian metric 
\begin{equation} \label{eq:vol-normalization}
\begin{split}
  \overline{g} &:= \vol(M,g)^{- \frac{2}{m}} \, g, 
\end{split}
\end{equation}
 then we have  $\vol(M,\overline{g}) \equiv 1.$
\end{lem}

\vspace{0.2cm}
 In particular, from the properties $(2)$ and $(3)$, we have 
\begin{equation} \label{eq:ev-normalization}
\begin{split}
  \lambda^{(p)}_k(M, \overline{g}) 
   &= \vol(M,g)^{\frac{2}{m}} \, \lambda^{(p)}_k(M,g)
\end{split}
\end{equation}
 for any $p$ and $k$.


\section{The hyperbolic dumbbell and its connected sum}

\subsection{The hyperbolic dumbbell}

 Following Boulanger and Courtois \cite{Boulanger-Courtois[22]}, Section $5$, pp.3626--3628, 
 we recall the $n$-dimensional hyperbolic dumbbell $(C_{\e}, g_{\e})$ 
 with parameter $\e >0$.

 For any $\e >0$, we first consider the $n$-dimensional hyperbolic cylinder 
 $C_{0,\e} := [-L,L] \times {\Bbb S}^{n-1}$ with the Riemannian metric 
\begin{equation} \label{eq:hyp-cylinder}
\begin{split}
   g_{\e} &= dr \oplus \e^2 \cosh^2(r) g_{{\Bbb S}^{n-1}} \quad (\e >0)
\end{split}
\end{equation}
 for $-L \le r \le L$, where $L:= |\log \e|$ ($\e=e^{-L}$) for short 
 and $g_{{\Bbb S}^{n-1}}$ denotes the standard Riemannian metric on 
 the $n-1$ dimensional standard sphere ${\Bbb S}^{n-1}$ of constant curvature one.

 Let $B_1, B_2$ be two $n$-dimensional spheres with the standard metrics 
 from which $n$-dimensional disks are removed.
 We glue $B_1, B_2$ to the boundary of this hyperbolic cylinder $C_{0,\e}$,
 identifying $\partial B_1$ with the left-side boundary $\{- L \} \times {\Bbb S}^{n-1}$
 and $\partial B_2$ with the right-side boundary $\{ L \} \times {\Bbb S}^{n-1}$. 
 It means that the removed disks on $B_1, B_2$ have the radius 
 $\e \cosh(|\log \e|) \rightarrow 1/2$ as $\e \to 0$.
 The resulting manifold is diffeomorphic to ${\Bbb S}^n$.
 We extend the Riemannian metric $g_{\e}$ on the hyperbolic cylinder $C_{0,\e}$
 to the whole Riemannian metric on ${\Bbb S}^n$ which is independent of $\e$ 
 on the both-sides $B_1, B_2$. 
 In addition, we can choose the extended Riemannian metric as the standard sphere 
 metrics on the both-sides $B_1, B_2$ away from their boundaries.
 We also denote by $g_{\e}$ this extended Riemannian metric, and we call 
 the resulting Riemannian manifold the $n$-dimensional hyperbolic dumbbell 
 denoted by $(C_{\e}, g_{\e})$ (see Figure $1$) .

\begin{figure}[H] \label{fig:hyp-dumbbell}
\begin{center}
\begin{tikzpicture}[scale=0.8]
  \draw (0,-0.1) arc [radius=0.1, start angle=-90, end angle=90];
  \draw[dotted] (0,0.1) arc [radius=0.1, start angle=90, end angle=270];
%
  \draw [domain=-3:3] plot (\x,{(cosh(\x))/10});
  \draw [domain=-3:3] plot (\x,{- (cosh(\x))/10});
  \draw[<->] (-3.0,-1.5) -- (3.0,-1.5);
  \draw (0,-1.4) -- (0,-1.6); 
  \draw (0,-1.5)  node[below] {$0$};  
  \draw (3.3,-1.5)  node[below] {$L = |\log \e|$}; 
  \draw (-3.0,-1.5) node[below] {$-L$}; 
  \draw[->] (0,1.0) -- (0,0.3);
  \draw[->] (0,-1.0) -- (0,-0.3);
  \draw (0.3,0.7) node {$\e$};
  \draw (3,-1) arc [radius=sqrt(2), start angle=-135, end angle=135];
  \draw (3,-1) arc [radius=sqrt(2), start angle=-45, end angle=45];
  \draw[dotted] (3,1) arc [radius=sqrt(2), start angle=135, end angle=225];
  \draw (4,0) node {$B_2$};
  \draw (-3,1) arc [radius=sqrt(2), start angle=45, end angle=315];
  \draw (-3,-1) arc [radius=sqrt(2), start angle=-45, end angle=45];
  \draw[dotted] (-3,1) arc [radius=sqrt(2), start angle=135, end angle=225];
  \draw (-4,0) node {$B_1$};
\end{tikzpicture}
 \vspace{-0.5cm}
\end{center}
   \caption{the hyperbolic dumbbell $(C_{\e}, g_{\e})$}
\end{figure}

 We precisely exhibit the way of connecting of $g_{\e}$ and 
 the standard metric of the sphere as follows:
 From the symmetry of the hyperbolic cylinder $C_{0,\e}$, 
 it is enough to consider the connecting part corresponding to
 $r= L = |\log \e|$.

 We introduce the new coordinate $s := r - L = r + \log \e$, then 
\begin{equation} \label{eq:func-f_e}
\begin{split}
  f_{\e}(s) &:= \e \cosh(s+L)
   = \dfrac{1}{2} e^s + \dfrac{\e^2}{2} e^{-s}
\end{split}
\end{equation}
 is the warping function of $g_{\e}$ on $- 2L \le s \le 0$.
 We set 
\begin{equation} \label{eq:func-h}
\begin{split}
   h(s) &:= \sin \big( s + \frac{\pi}{6} \big)
   \quad  (0 \le s \le \frac{\pi}{12}).
\end{split}
\end{equation}
 To connect these two positive functions $f_{\e}(s)$ and $h(s)$ smoothly, 
 we define the new function $F_{\e}(s)$ as follows:
\begin{equation} \label{eq:smooth-connect}
\begin{split}
  F_{\e}(s) &:= \chi(s) f_{\e}(s) + \big( 1 - \chi(s) \big) h(s) 
   \quad  (0 \le s \le \frac{\pi}{12}),
\end{split}
\end{equation}
 where $\chi(s)$ is a smooth cut-off function satisfying 
\begin{equation*} 
\begin{split}
  \chi(s) &= 
\begin{cases}
  \  1   &   (0 \le s \le \frac{\pi}{36}), \\
  \  0   &   (\frac{\pi}{18} \le s \le \frac{\pi}{12}).
\end{cases}
\end{split}
\end{equation*}
 By \eqref{eq:func-f_e} and \eqref{eq:func-f_e}, 
 the equation \eqref{eq:smooth-connect} is written as 
\begin{equation*} 
\begin{split}
  F_{\e}(s) &= \left\{ \chi(s) \dfrac{1}{2} e^s + \big( 1 - \chi(s) \big) h(s) \right\} 
         + \dfrac{\e^2}{2} e^{-s} \chi(s).
\end{split}
\end{equation*}
 If we take $\e$ small enough, the term $\frac{\e^2}{2} e^{-s} \chi(s)$
 and its derivatives are also small enough. 
 Hence, there exists an $\e_0 >0$ such that for all $0< \e < \e_0$, 
 $F_{\e}(s)$, $F^{\prime}_{\e}(s)$ and $F^{\prime \prime}_{\e}(s)$ 
 are uniformly bounded on $0 \le s \le \frac{\pi}{12}$.
 In particular, since $f_{\e}(s)$ and $h(s)$ are monotone increasing, 
 we see 
\begin{equation*} 
\begin{split}
  0.5 &< f_{\e}(0) \le f_{\e}(s) \le f_{\e}(\tfrac{\pi}{12}) < e^{\pi/12} < 0.7, \\
  0.5 &= \sin ( \tfrac{\pi}{6}) \le h(s) \le \sin ( \tfrac{\pi}{4}) 
    = \tfrac{1}{\sqrt{2}} < 0.8.
\end{split}
\end{equation*}
 Thus, we have
\begin{equation} \label{eq:F_e-bdd}
\begin{split} 
   0.5 &\le F_{\e}(s) < 0.8 
     \quad  (0 \le s \le \tfrac{\pi}{12}).
\end{split}
\end{equation}

 Now, if we take a Riemannian metric around the connecting part as 
\begin{equation} \label{eq:connecting-metric}
\begin{split}
  ds^2 \oplus F^2_{\e}(s) \, g_{{\Bbb S}^{n-1}}  \quad 
      (0 \le s \le \tfrac{\pi}{12}), 
\end{split}
\end{equation}
 then the whole Riemannian metric $g_{\e}$ on the hyperbolic dumbbell $C_{\e}$ 
 is smooth, and coincides with the Riemannian metric on the hyperbolic cylinder $C_{0, \e}$
 and coincides with the standard sphere metric on $B_2$. 
 In fact, since $F_{\e}(s) = f_{\e}(s) = \e \cosh (s)$ for $0 \le s \le \tfrac{\pi}{36}$,
 we have 
\begin{equation*}
\begin{split}
  ds^2 \oplus F^2_{\e}(s) \, g_{{\Bbb S}^{n-1}} 
     &= ds^2 \oplus \e^2 \cosh^2 (s) \, g_{{\Bbb S}^{n-1}} \
   \text{ on } [0, \tfrac{\pi}{36}] \times {\Bbb S}^{n-1},
\end{split}
\end{equation*}
 which coincides with the Riemannian metric on the hyperbolic cylinder.
 Since $F_{\e}(s) = h(s) = \sin (s+ \tfrac{\pi}{6})$ for 
 $\tfrac{\pi}{18} \le s \le \tfrac{\pi}{12}$, we have 
\begin{equation*}
\begin{split}
  ds^2 \oplus F^2_{\e}(s) \, g_{{\Bbb S}^{n-1}} 
   &= ds^2 \oplus \sin^2 (s+ \tfrac{\pi}{6}) \, g_{{\Bbb S}^{n-1}} 
     \ \text{ on } [\tfrac{\pi}{18}, \tfrac{\pi}{12}] \times {\Bbb S}^{n-1},
\end{split}
\end{equation*}
 which coincides with the standard sphere metric on $B_2$.

\begin{lem}[Sectional curvature of warped product manifolds]
  \label{lem:sect-curv-warped-product}
 For a Riemannian manifold $(N,h)$ and a smooth positive function $f(r)$
 on the interval $I$, 
 we consider the warped product manifold 
  $(M, g_f) := (I \times N, dr^2 \oplus f^2(r) h)$. 
 For orthonormal vectors $X$ and $Y$ on $(N,h)$,
 the vectors $\widetilde{X} := f(r)^{-1} X$, $\widetilde{Y} := f(r)^{-1} Y$
 on $M$ are orthonormal and perpendicular to $\partial_r = \frac{\partial}{\partial r}$
 with respect to the metric $g_f$.

 Then, the sectional curvatures $K_M$ of $(M, g_f)$ are given as follows:
\begin{enumerate}
 \renewcommand{\labelenumi}{$(\roman{enumi})$}
 \item  $K_M (\partial_r, \widetilde{X})
     = - \dfrac{f^{\prime \prime}(r) }{f(r)},$  
 \item  $ K_M (\widetilde{X}, \widetilde{Y}) 
     = \dfrac{K_N(X,Y) - ( f^{\prime}(r))^2 }{f^2(r)}.$
\end{enumerate}
 In particular, if $(N^n,h) = ({\Bbb S}^n, g_{{\Bbb S}^n})$, then $K_N(X,Y) \equiv 1.$
\end{lem}

 For the proof of this lemma, see Petersen \cite{Petersen[16]}, {\bf 4.2.3}, p.121.

\begin{lem} \label{lem:curv-bdd-below}
 The sectional curvature $K_{C_{\e}}$ on the hyperbolic dumbbell 
 $(C_{\e}, g_{\e})$ is uniformly bounded below in $\e$.
 That is, there exists some positive constant $\kappa^{\prime} > 0$ 
 independent of $\e$ such that 
\begin{equation*}
\begin{split}
       K_{C_{\e}} &\ge - \kappa^{\prime}.
\end{split}
\end{equation*}
\end{lem}

\begin{proof}
 Since the metric on the both-sides bumps is independent of $\e$,
 we have only to show the boundedness on the central part of $C_{\e}.$

 We use the same notation as in Lemma \ref{lem:sect-curv-warped-product}.
 On the hyperbolic cylinder $C_{0,\e}$, from \eqref{eq:hyp-cylinder}, 
 we have 
\begin{equation*}
\begin{split}
 K(\partial_r, \widetilde{X}) 
   &= - \dfrac{ \e \cosh^{\prime \prime}(r)}{ \e \cosh(r) }
    = - \dfrac{\cosh(r)}{\cosh(r)} = -1, \\[0.2cm]
  K(\widetilde{X}, \widetilde{Y}) 
    &=  \dfrac{1 - \big(  \e \cosh^{\prime}(r) \big)^2 }{ 
          \big( \e \cosh(r) \big)^2 } 
   \ge  - \dfrac{ \e^{2} \sinh^{2}(r) }{\e^{2} \cosh^2(r) }  
    =   - \dfrac{ \sinh^2(r) }{\cosh^2(r)}
   \ge  -1.
\end{split}
\end{equation*}
 Thus, the sectional curvature on the hyperbolic cylinder $C_{0,\e}$
 is uniformly bounded below by $-1$.

 Next, around the right-side of the boundary $\partial C_{0, \e}$, 
 since the Riemannian metric is expressed as \eqref{eq:connecting-metric}, 
 there exist positive constants $\kappa_1, \kappa_2 >0$ independent of $\e$ 
 such that 
\begin{equation*}
\begin{split}
 K(\partial_r, \widetilde{X}) 
   &= - \dfrac{ F^{\prime \prime}_{\e}(s)}{ F_{\e}(s) }
    \ge - \kappa_1,  \\[0.2cm]
 K(\widetilde{X}, \widetilde{Y}) 
    &=  \dfrac{1 - \big( F^{\prime}_{\e}(s) \big)^2 }{ F^2_{\e}(s) }
   \ge  - \kappa_2.
\end{split}
\end{equation*}

 Therefore, we find that the sectional curvature on the hyperbolic dumbbell
 $(C_{\e}, g_{\e})$ is uniformly bounded below in $\e$.
\end{proof}

 Furthermore, the volume of the hyperbolic dumbbell $(C_{\e}, g_{\e})$ 
 is uniformly bounded in $\e$.

\begin{lem} \label{lem:vol-finite}
 There exist two positive constants $V_1, V_2 >0$ independent of $\e$
 such that 
\begin{equation} \label{eq:volume-estimate}
\begin{split}
   0< V_1 &\le \vol(C_{\e}, g_{\e}) \le V_2.
\end{split}
\end{equation}
\end{lem}

\begin{proof}
 We first estimate the volume of the $n$-dimensional hyperbolic cylinder 
 $(C_{0,\e}, g_{\e})$ from above. 
 The volume of $(C_{0,\e}, g_{\e})$ is
\begin{equation*} \label{eq:vol-hyp-cylinder-1}
\begin{split}
 \vol(C_{0,\e}, g_{\e}) &= 2 \, \dint^L_0 \dint_{{\Bbb S}^{n-1}}
     \big( \e \cosh (r) \big)^{n-1} dr d \mu_{g_{{\Bbb S}^{n-1}}} \\
   &= 2 \vol({\Bbb S}^{n-1}) \, \e^{n-1} \, \dint^L_0 \cosh^{n-1}(r) dr.
\end{split}
\end{equation*}
 Since $\cosh(r)\leq e^r$ for $r\geq 0$ and $L = - \log \e$, we have 
\begin{equation} \label{eq:int-cosh^n}
\begin{split}
  \dint^L_0 \cosh^{n-1}(r) dr 
   &\le  \dint^L_0 e^{(n-1)r}dr  = \dfrac{1}{n-1}(e^{(n-1)L}-1) \\
   &\le \dfrac{1}{n-1}e^{(n-1)L} = \dfrac{1}{n-1}\e^{-(n-1)},
\end{split}
\end{equation}
 and finally
\begin{equation*} 
\begin{split}
  \vol(C_{0,\e}, g_{\e}) &\le \dfrac{2}{n-1}\vol({\Bbb S}^{n-1}).
\end{split}
\end{equation*}
 Thus, the volume of $(C_{0,\e}, g_{\e})$ is finite for $\e$.
 Since the volumes of $B_1$ and $B_2$ are bounded above in $\e$,
 the total volume of the hyperbolic dumbbell $(C_{\e}, g_{\e})$ 
 is unform bounded above in $\e$.

 On the other hand, since the metrics on $B_1$ and $B_2$ away from 
 their boundaries coincide with the standard sphere metrics, 
 there exists a uniform lower bound of the volume:
\begin{equation*} 
\begin{split}
  \vol(C_{\e}, g_{\e}) &\ge  \vol(B_1) + \vol(B_2) 
    \ge  \frac{1}{2} \vol({\Bbb S}^n) + \frac{1}{2} \vol({\Bbb S}^n) 
    = \vol({\Bbb S}^n) := V_1.
\end{split}
\end{equation*}
\end{proof}


\subsection{The connected sum of $k$-copies of the hyperbolic dumbbell}

\label{sub-sect:k-hyp-dumbbells}

 Next, we perform the connected sum of $k$-copies of the hyperbolic dumbbell 
 in series. The resulting Riemannian manifold is denoted by 
 $C_{k,\e} = \overset{k}{\sharp} C_{\e}$ with the periodic metric $g_{C_{k,\e}}$
 (see the Figure $2$).

\vspace{0.2cm}

\begin{figure}[H] \label{fig:k-copies-hyp-dumbbell}
\begin{center}
\begin{tikzpicture}[scale=0.42]
  \draw (0,-0.1) arc [radius=0.1, start angle=-90, end angle=90];
  \draw[dotted] (0,0.1) arc [radius=0.1, start angle=90, end angle=270];
  \draw [domain=-3:3] plot (\x,{(cosh(\x))/10});
  \draw [domain=-3:3] plot (\x,{- (cosh(\x))/10});
  \draw[<->] (-3.0,-1.5) -- (3.0,-1.5);
  \draw (0,-1.4) -- (0,-1.6); 
  \draw (0,-1.5)  node[below] {$0$};  
  \draw (3.0,-1.5)  node[below] {$L$}; 
  \draw (-3.0,-1.5) node[below] {$-L$}; 
  \draw[->] (0,1.5) -- (0,0.3);
  \draw[->] (0,-1.2) -- (0,-0.3);
  \draw (0.6,1.0) node {$\e$};
  \draw (-3,1) arc [radius=sqrt(2), start angle=45, end angle=315];
  \draw (-3,-1) arc [radius=sqrt(2), start angle=-45, end angle=45];
  \draw[dotted] (-3,1) arc [radius=sqrt(2), start angle=135, end angle=225];
  \draw (-4,0) node {$B_1$};
  \draw (3,-1) arc [radius=sqrt(2), start angle=-135, end angle=135];
  \draw (3,-1) arc [radius=sqrt(2), start angle=-45, end angle=45];
  \draw[dotted] (3,1) arc [radius=sqrt(2), start angle=135, end angle=225];
  \draw[dotted] (5,1) arc [radius=sqrt(2), start angle=135, end angle=225];
  \draw (4.1,0) node {$B_2$};
  \draw [domain=-3:3] plot ({\x +8},{(cosh(\x))/10});
  \draw [domain=-3:3] plot ({\x +8},{- (cosh(\x))/10});
  \draw (8,-0.1) arc [radius=0.1, start angle=-90, end angle=90];
  \draw[dotted] (8,0.1) arc [radius=0.1, start angle=90, end angle=270];
  \draw (11,-1) arc [radius=sqrt(2), start angle=-135, end angle=135];
  \draw (11,-1) arc [radius=sqrt(2), start angle=-45, end angle=45];
  \draw[dotted] (11,1) arc [radius=sqrt(2), start angle=135, end angle=225];
  \draw[dotted] (13,1) arc [radius=sqrt(2), start angle=135, end angle=225];
  \draw (12.1,0) node {$B_3$};
  \draw [domain=-3:-1] plot ({\x +16},{(cosh(\x))/10});
  \draw [domain=-3:-1] plot ({\x +16},{- (cosh(\x))/10});
  \draw[thick, dotted] (15.2,0) -- (16.2,0);
  \draw [domain=1:3] plot ({\x +15.5},{(cosh(\x))/10});
  \draw [domain=1:3] plot ({\x +15.5},{- (cosh(\x))/10});
  \draw (18.5,-1) arc [radius=sqrt(2), start angle=-135, end angle=135];
  \draw (18.5,-1) arc [radius=sqrt(2), start angle=-45, end angle=45];
  \draw[dotted] (18.5,1) arc [radius=sqrt(2), start angle=135, end angle=225];
  \draw[dotted] (20.5,1) arc [radius=sqrt(2), start angle=135, end angle=225];
  \draw (19.6,0) node {$B_{k}$};
  \draw [domain=-3:3] plot ({\x +23.5},{(cosh(\x))/10});
  \draw [domain=-3:3] plot ({\x +23.5},{- (cosh(\x))/10});
  \draw (26.5,-1) arc [radius=sqrt(2), start angle=-135, end angle=135];
  \draw (26.5,-1) arc [radius=sqrt(2), start angle=-45, end angle=45];
  \draw[dotted] (26.5,1) arc [radius=sqrt(2), start angle=135, end angle=225];
  \draw (27.9,0) node {$B_{k+1}$};
\end{tikzpicture}
\end{center}
   \caption{the connected sum of $k$-copies of the hyperbolic dumbbell}
\end{figure}
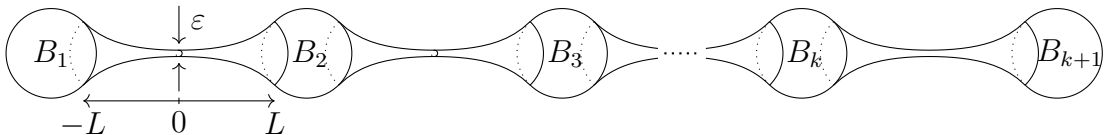

\vspace{-0.2cm}

 From the construction, $(C_{k,\e}, g_{C_{k,\e}})$ also satisfies 
 the following property:

\begin{lem} \label{lem:properties-k-hyp-dumbbells}
\begin{enumerate} 
 \renewcommand{\labelenumi}{$(\roman{enumi})$}
 \item  The sectional curvature of $(C_{k,\e}, g_{C_{k,\e}})$ is uniformly 
 bounded below in $\e$;
 \item  The volume of $(C_{k,\e}, g_{C_{k,\e}})$ is uniformly bounded in $\e$.
\end{enumerate}
\end{lem}


\section{Construction of the Riemannian metrics}

 We construct a one-parameter family of Riemannian metrics 
 $\{ \overline{g}_{\e} \}_{\e >0}$ on any closed manifold $M$ 
 with the volume one and the sectional curvature uniformly bounded below:
 $K_{\overline{g}_{\e}} \ge - \kappa$, where $\kappa >0$ is independent of $\e$.

\vspace{0.2cm}
 Let $M$ be a connected oriented closed $C^{\infty}$-manifold of dimension $m \ge 2$.
 For a given degree $p$ with $0 \le p \le m-2$, 
 we can take an embedded $p$-dimensional sphere ${\Bbb S}^p$ into $M$ 
 whose normal bundle is trivial.
 Then, a closed tubular neighborhood $\Tub({\Bbb S}^p)$ of ${\Bbb S}^p$ in $M$ 
 can be identified with 
\begin{equation} \label{eq:tubular-nbd}
\begin{split}
   \Tub({\Bbb S}^p) &\cong {\Bbb S}^p \times {\Bbb D}^{m-p},
\end{split}
\end{equation}
 where ${\Bbb D}^n$ denotes the $n$-dimensional closed unit disk in $\R^n$.

 We now take any Riemannian metric $g_{p,M}$ on $M$ such that 
 $g_{p,M}$ on $\Tub({\Bbb S}^p)$ is the product metric of 
 $g_{{\Bbb S}^p}$ on ${\Bbb S}^p$ and the standard Euclidean metric $g_{\R^{m-p}}$ 
 on ${\Bbb D}^{m-p}$:
\begin{equation} \label{eq:product-tubular-nbd}
\begin{split}
  g_{p,M} &=  g_{{\Bbb S}^p} \oplus g_{\R^{m-p}}
   \quad  \text{ on } \Tub({\Bbb S}^p) = {\Bbb S}^p \times {\Bbb D}^{m-p}.
\end{split}
\end{equation}
  We decompose $M$ into the two components $H_1, H_2$: 
\begin{equation} \label{eq:decomp-H}
\begin{split}
  H_1 &:=  {\Bbb S}^{p} \times {\Bbb D}^{m-p},  \\
  H_2 &:=  \overline{M \setminus H_1} 
       =   \overline{M \setminus ({\Bbb S}^{p} \times {\Bbb D}^{m-p})}.
\end{split}
\end{equation}
 Then, while fixing the metric $g_{p,M}$ on $H_2$, 
 we change the metric $g_{p,M}$ on $H_1$ to a new metric.

 For any real number $\e >0$ and any natural number $k \ge 1$,
 as constructed in the previous sub-section {\bf \ref{sub-sect:k-hyp-dumbbells}},
 we take the connected sum of $k$-copies of the $(m-p)$-dimensional 
 hyperbolic dumbbell $C_{k,\e} = \overset{k}{\large \sharp} C_{\e}$ 
 with the Riemannian metrics $g_{C_{k,\e}}$, 
 and glue it to the second factor ${\Bbb D}^{m-p}$ of $H_1$
 (See the Figure $3$).
 This gluing can be done independently of $\e$.
 We also use the same notation of this new Riemannian metrics $g_{C_{k,\e}}$
 on the gluing $C_{k,\e} \sharp {\Bbb D}^{m-p} \cong {\Bbb D}^{m-p}_{\e}$.

\vspace{-0.3cm}

\begin{figure}[H] \label{fig:k-copies-hyp-dumbbells}
\begin{center}
\begin{tikzpicture}[scale=0.45]
  \draw (0,-0.1) arc [radius=0.1, start angle=-90, end angle=90];
  \draw[dotted] (0,0.1) arc [radius=0.1, start angle=90, end angle=270];
  \draw [domain=-3:3] plot (\x,{(cosh(\x))/10});
  \draw [domain=-3:3] plot (\x,{- (cosh(\x))/10});
  \draw[<->] (-3.0,-1.5) -- (3.0,-1.5);
  \draw (0,-1.4) -- (0,-1.6); 
  \draw (0,-1.5)  node[below] {$0$};  
  \draw (3.0,-1.5)  node[below] {$L$}; 
  \draw (-3.0,-1.5) node[below] {$-L$}; 
  \draw[->] (0,1.5) -- (0,0.3);
  \draw[->] (0,-1.2) -- (0,-0.3);
  \draw (0.6,1.0) node {$\e$};
  \draw (-3,1) arc [radius=sqrt(2), start angle=45, end angle=315];
  \draw (-3,-1) arc [radius=sqrt(2), start angle=-45, end angle=45];
  \draw[dotted] (-3,1) arc [radius=sqrt(2), start angle=135, end angle=225];
  \draw (-4,0) node {$B_1$};
  \draw (3,-1) arc [radius=sqrt(2), start angle=-135, end angle=135];
  \draw (3,-1) arc [radius=sqrt(2), start angle=-45, end angle=45];
  \draw[dotted] (3,1) arc [radius=sqrt(2), start angle=135, end angle=225];
  \draw[dotted] (5,1) arc [radius=sqrt(2), start angle=135, end angle=225];
  \draw (4,0) node {$B_2$};
  \draw [domain=-3:3] plot ({\x +8},{(cosh(\x))/10});
  \draw [domain=-3:3] plot ({\x +8},{- (cosh(\x))/10});
  \draw (8,-0.1) arc [radius=0.1, start angle=-90, end angle=90];
  \draw[dotted] (8,0.1) arc [radius=0.1, start angle=90, end angle=270];
  \draw (11,-1) arc [radius=sqrt(2), start angle=-135, end angle=135];
  \draw (11,-1) arc [radius=sqrt(2), start angle=-45, end angle=45];
  \draw[dotted] (11,1) arc [radius=sqrt(2), start angle=135, end angle=225];
  \draw[dotted] (13,1) arc [radius=sqrt(2), start angle=135, end angle=225];
  \draw (12,0) node {$B_3$};
  \draw [domain=-3:-1] plot ({\x +16},{(cosh(\x))/10});
  \draw [domain=-3:-1] plot ({\x +16},{- (cosh(\x))/10});
  \draw[thick, dotted] (15.2,0) -- (16.2,0);
  \draw [domain=1:3] plot ({\x +15.5},{(cosh(\x))/10});
  \draw [domain=1:3] plot ({\x +15.5},{- (cosh(\x))/10});
  \draw (18.5,-1) arc [radius=sqrt(2), start angle=-135, end angle=135];
  \draw (18.5,-1) arc [radius=sqrt(2), start angle=-45, end angle=45];
  \draw[dotted] (18.5,1) arc [radius=sqrt(2), start angle=135, end angle=225];
  \draw[dotted] (20.5,1) arc [radius=sqrt(2), start angle=135, end angle=225];
  \draw (19.5,0) node {$B_k$};
  \draw [domain=-3:3.2] plot ({\x +23.5},{(cosh(\x))/10});
  \draw [domain=-3:3.2] plot ({\x +23.5},{- (cosh(\x))/10});
  \draw (23.5,-0.1) arc [radius=0.1, start angle=-90, end angle=90];
  \draw[dotted] (23.5,0.1) arc [radius=0.1, start angle=90, end angle=270];
  \draw (26.7,-1.2) arc (-55:55:0.5cm and 1.5cm);
  \draw[dotted] (26.7,1.2) arc (130:220:0.5cm and 1.5cm);
  \draw (26.2,-0.7) arc (-165:160:0.8cm and 2.5cm);
  \draw[dotted] (26.3,0.8) arc (150:190:0.8cm and 2.5cm);
  \draw (27.5,3.2) node {${\Bbb D}^{m-p}$};
\end{tikzpicture}
 \vspace{-1.2cm}
\end{center}
   \caption{gluing $C_{k,\e}$ to ${\Bbb D}^{m-p}$}
\end{figure}

\vspace{-0.3cm}

 Thus, we obtain a one-parameter family of Riemannian metrics $g_{\e,p,k}$
 on  $H_1 = \Tub({\Bbb S}^p) = {\Bbb S}^p \times {\Bbb D}^{m-p}_{\e}$ as 
\begin{equation}  \label{eq:metric-H_1}
\begin{split}
  g_{\e,p,k} &:= g_{{\Bbb S}^p} \oplus g_{C_{k,\e}}.
\end{split}
\end{equation}
 Then, we define the one-parameter family of Riemannian metrics 
 $\{ g_{\e,p,k} \}_{\e >0}$ on $M$ as 
\begin{equation} \label{eq:metric-g_e}
\begin{split}
  g_{\e,p,k} &:= 
 \begin{cases}
   g_{{\Bbb S}^p} \oplus g_{C_{k,\e}}  &  \text{ on } 
        H_1 = {\Bbb S}^{p} \times {\Bbb D}^{m-p}_{\e}, \\
   \quad  g_{p,M}   &  \text{ on } H_2 =  \overline{M \setminus H_1}.
 \end{cases}
\end{split}
\end{equation}

\vspace{0.2cm}
 Finally, we normalize this metric $g_{\e,p,k}$ whose total volume is one.
 That is, we define 
\begin{equation} \label{eq:normalized-g_{e,p,k}}
\begin{split}
 \overline{g}_{\e,p,k}
    &:= \vol(M,g_{\e,p,k})^{- \frac{2}{m}} \, g_{\e,p,k} \ \text{ on } M.
\end{split}
\end{equation}
 Then, $\vol(M, \overline{g}_{\e,p,k}) \equiv 1$.

\begin{lem}[Estimates of the volume] \label{lem:M-vol-estimate}
 There exist positive constants $V, A, B >0$ independent of 
 $\e$ and $k$ such that 
\begin{equation} \label{eq:M-vol-estimate}
\begin{split}
   0< V &\le \vol(M, g_{\e,p,k}) \le A k + B.
\end{split}
\end{equation}
\end{lem}

\begin{proof}
 For an upper bound, from Lemma \ref{lem:vol-finite}, we have 
\begin{equation*} \label{eq:vol-upper-bdd}
\begin{split}
 \vol(M, g_{\e,p,k}) 
  &= \vol(H_1, g_{\e,p,k}) + \vol(H_2, g_{\e,p,k}) \\
  &= \vol({\Bbb S}^p) \cdot \vol(C_{k,\e}, g_{C_{k,\e}})
        + \vol(H_2, g_{p,M}) \\
  &\le  \vol({\Bbb S}^p) \cdot \vol(C_{1,\e}, g_{C_{1,\e}}) \, k
        + \vol(H_2, g_{p,M})  \\
  &\le  A k + B,
\end{split}
\end{equation*}
 where $A, B >0$ are some constants independent of $\e$ and $k$.

 For a positive lower bound, from Lemma \ref{lem:vol-finite}, 
 we also have
\begin{equation*} \label{eq:vol-lower-bdd}
\begin{split}
 \vol(M, g_{\e,p,k}) 
  &= \vol(H_1, g_{\e,p,k}) + \vol(H_2, g_{\e,p,k}) \\
  &\ge \vol(H_1, g_{\e,p,k})
   =   \vol({\Bbb S}^p) \cdot \vol(C_{k,\e}, g_{C_{k,\e}}) \\
  &\ge  \vol({\Bbb S}^p) \cdot \vol(C_{1,\e}, g_{C_{1,\e}})
   \ge  \vol({\Bbb S}^p) \cdot V_1 > 0.
\end{split}
\end{equation*}
\end{proof}

 Hence, from the property of the sectional curvature 
\begin{equation*} \label{eq:sect-curv-scaling}
\begin{split}
  K_{(M, \overline{g}_{\e,p,k})}
   &= \vol(M,g_{\e,p,k})^{\frac{2}{m}} \, K_{(M, g_{\e,p,k})}
\end{split}
\end{equation*}
 and Lemma \ref{lem:M-vol-estimate}, we find that 
 the family of volume-normalized Riemannian metrics
 $\{ \overline{g}_{\e,p,k} \}_{\e >0}$ on $M$ defined in 
 \eqref{eq:normalized-g_{e,p,k}} satisfies the same properties
 as in Lemma \ref{lem:properties-k-hyp-dumbbells}.

\begin{lem} \label{lem:properties}
\begin{enumerate} 
 \renewcommand{\labelenumi}{$(\roman{enumi})$}
 \item  The sectional curvature of $(M, \overline{g}_{\e,p,k})$ 
  is uniformly bounded below in $\e$;
 \item  The volume of $(M, \overline{g}_{\e,p,k})$ is identically one.
\end{enumerate}
\end{lem}


\section{The proof of Theorem \ref{thm:main}}

 We give the proof of Theorem \ref{thm:main} by using the 
 min-max principle for the Hodge-Laplacian acting on co-exact forms.
 We denote by $\lambda^{\prime (p)}_k(M,g)$ and $\lambda^{\prime \prime (p)}_k(M,g)$
 the $k$-th eigenvalues of the Hodge-Laplacian acting on exact and 
 co-exact forms, respectively, which are always positive.
 Theorem \ref{thm:main} is a corollary of Lemma \ref{lem:k-ev-estimate}.

\begin{lem}[Small eigenvalues] \label{lem:k-ev-estimate}
 Let $p$ be an integer with $0 \le p \le m-2$.
 For any integer $k \ge 1$ and any real number $\e >0$, there exists
 a positive constant $C(m,p,\overline{k}) >0$ independent of $\e$ such that 
\begin{equation*}
\begin{split}
  \lambda^{\prime \prime (p)}_k (M, \overline{g}_{\e,p,\overline{k}})
       &\le  \dfrac{C(m,p,\overline{k})}{| \log \e|^2},
\end{split}
\end{equation*}
 where $\overline{k} := k + b_p(M)$, where $b_p(M)$ is the $p$-th 
 Betti number of $M$.
\end{lem}

\begin{rem} \label{rem:duality}
 In the case of $p=m-1, m$, by the Hodge duality 
 $\lambda^{\prime \prime (p)}_k = \lambda^{\prime (m-p)}_k$, 
 we can reduce to the case of $p=1,0$, respectively.
 Therefore, the same statement as Lemma \ref{lem:k-ev-estimate} 
 for exact $p$-forms still holds in the case of $p=m-1,m$.
\end{rem}

\begin{proof}
 To prove the estimate in Lemma \ref{lem:k-ev-estimate}, 
 we use the min-max principle for the Hodge-Laplacian acting on 
 co-exact $p$-forms. 
 Since the space of co-closed $p$-forms modulo co-exact $p$-forms 
 is that of harmonic $p$-forms, whose dimension is $b_p(M)$, 
 we may construct $\overline{k} = k + b_p(M)$ test co-closed $p$-forms 
 $\varphi_i$ on $M$.

 Let $\chi_i$ be the linear cut-off functions on $C_{\overline{k},\e}$ 
 as follows (See Figure $4$):
 For $\chi_1$, we define $\chi_1$ as 
\begin{equation*} \label{eq:linear-cut-off-func-1}
\begin{split}
  \chi_1  &:=
 \begin{cases}
    \quad  1  &  \text{for } r \le - L, \\[0.2cm]
     - \dfrac{r}{L} & \text{for } - L \le r \le 0, \\[0.2cm]
    \quad  0  &   \text{for } 0 \le r. 
 \end{cases}
\end{split}
\end{equation*}
 For $\chi_i$ $(i=2,3,\dots,\overline{k})$, we define $\chi_i$ periodically as 
\begin{equation*} \label{eq:linear-cut-off-func-i}
\begin{split}
  \chi_i  &:=
 \begin{cases}
  \quad  0  \quad   \text{for } r \le (2i-4)L + (i-2) \tfrac{2}{3} \pi, \\[0.2cm]
     \  \dfrac{1}{L} \big( r - (2i-4)L - (i-2) \tfrac{2}{3} \pi \big)   \\
  \qquad  \ \, \text{for } (2i-4)L + (i-2) \tfrac{2}{3} \pi \le r 
       \le (2i-3)L + (i-2) \tfrac{2}{3} \pi, \\
  \quad  1  \quad  \text{for } (2i-3)L + (i-2) \tfrac{2}{3} \pi \le r 
       \le (2i-3)L + (i-1) \tfrac{2}{3} \pi, \\[0.2cm]
   - \dfrac{1}{L} \big( r - (2i-2)L - (i-1) \tfrac{2}{3} \pi \big)  \\
  \qquad \ \, \text{for } (2i-3)L + (i-1) \tfrac{2}{3} \pi \le r 
       \le  (2i-2)L + (i-1) \tfrac{2}{3} \pi, \\[0.2cm]
   \quad  0  \quad  \text{for } (2i-2)L + (i-1) \tfrac{2}{3} \pi \le r.
 \end{cases}
\end{split}
\end{equation*}
 From the construction, the interiors of the supports $\supp(\chi_i)^{\circ}$
 are mutually disjoint: 
\begin{equation*} 
\begin{split}
  \supp(\chi_i)^{\circ} \cap \supp(\chi_j)^{\circ} = \emptyset  \quad (i \neq j).
\end{split}
\end{equation*}

\vspace{-0.5cm}

\begin{figure}[H] \label{fig:cut-off-funcs}
\begin{center}
\begin{tikzpicture}[scale=0.4]
  \draw (0,-0.1) arc [radius=0.1, start angle=-90, end angle=90];
  \draw[dotted] (0,0.1) arc [radius=0.1, start angle=90, end angle=270];
  \draw [domain=-3:3] plot (\x,{(cosh(\x))/10});
  \draw [domain=-3:3] plot (\x,{- (cosh(\x))/10});
  \draw[<->] (-3.0,-1.5) -- (3.0,-1.5);
  \draw (0,-1.4) -- (0,-1.6); 
  \draw (0,-1.5)  node[below] {$0$};  
  \draw (3.0,-1.5)  node[below] {$L$}; 
  \draw (-3.0,-1.5) node[below] {$-L$}; 
  \draw[->] (0,1.5) -- (0,0.3);
  \draw[->] (0,-1.2) -- (0,-0.3);
  \draw (0.6,1.0) node {$\e$};
  \draw (-3,1) arc [radius=sqrt(2), start angle=45, end angle=315];
  \draw (-3,-1) arc [radius=sqrt(2), start angle=-45, end angle=45];
  \draw[dotted] (-3,1) arc [radius=sqrt(2), start angle=135, end angle=225];
  \draw (-4,0) node {$B_1$};
  \draw (3,-1) arc [radius=sqrt(2), start angle=-135, end angle=135];
  \draw (3,-1) arc [radius=sqrt(2), start angle=-45, end angle=45];
  \draw[dotted] (3,1) arc [radius=sqrt(2), start angle=135, end angle=225];
  \draw[dotted] (5,1) arc [radius=sqrt(2), start angle=135, end angle=225];
  \draw (4,0) node {$B_2$};
  \draw [domain=-3:3] plot ({\x +8},{(cosh(\x))/10});
  \draw [domain=-3:3] plot ({\x +8},{- (cosh(\x))/10});
  \draw (8,-0.1) arc [radius=0.1, start angle=-90, end angle=90];
  \draw[dotted] (8,0.1) arc [radius=0.1, start angle=90, end angle=270];
  \draw (11,-1) arc [radius=sqrt(2), start angle=-135, end angle=135];
  \draw (11,-1) arc [radius=sqrt(2), start angle=-45, end angle=45];
  \draw[dotted] (11,1) arc [radius=sqrt(2), start angle=135, end angle=225];
  \draw[dotted] (13,1) arc [radius=sqrt(2), start angle=135, end angle=225];
  \draw (12,0) node {$B_3$};
  \draw [domain=-3:-1] plot ({\x +16},{(cosh(\x))/10});
  \draw [domain=-3:-1] plot ({\x +16},{- (cosh(\x))/10});
  \draw[thick, dotted] (15.2,0) -- (16.2,0);
  \draw [domain=1:3] plot ({\x +15.5},{(cosh(\x))/10});
  \draw [domain=1:3] plot ({\x +15.5},{- (cosh(\x))/10});
  \draw (18.5,-1) arc [radius=sqrt(2), start angle=-135, end angle=135];
  \draw (18.5,-1) arc [radius=sqrt(2), start angle=-45, end angle=45];
  \draw[dotted] (18.5,1) arc [radius=sqrt(2), start angle=135, end angle=225];
  \draw[dotted] (20.5,1) arc [radius=sqrt(2), start angle=135, end angle=225];
  \draw (19.5,0) node {$B_{\overline{k}}$};
  \draw [domain=-3:3.2] plot ({\x +23.5},{(cosh(\x))/10});
  \draw [domain=-3:3.2] plot ({\x +23.5},{- (cosh(\x))/10});
  \draw (23.5,-0.1) arc [radius=0.1, start angle=-90, end angle=90];
  \draw[dotted] (23.5,0.1) arc [radius=0.1, start angle=90, end angle=270];
  \draw (26.7,-1.2) arc (-55:55:0.5cm and 1.5cm);
  \draw[dotted] (26.7,1.2) arc (130:220:0.5cm and 1.5cm);
  \draw (26.2,-0.7) arc (-165:160:0.8cm and 2.5cm);
  \draw[dotted] (26.3,0.8) arc (150:190:0.8cm and 2.5cm);
  \draw (27.5,3.2) node {${\Bbb D}^{m-p}$};
%
  \draw (-5.5,-7.0) -- (27.5,-7.0);
  \draw (-5.5,-5) -- (-3,-5);
  \draw (-5.5,-5) -- (-3,-5);
  \draw (-3,-5) -- (0,-7);
  \draw[thick, dotted] (-3,-5) -- (-3,-7);
  \draw (-4,-4) node {$\chi_1$};
  \draw (-3,-7) node[below] {{\footnotesize $-L$}};
  \draw (0,-7)  node[below] {$0$};
  \draw (0,-7) -- (3,-5);
  \draw (3,-5) -- (5,-5);
  \draw (5,-5) -- (8,-7);
  \draw[thick, dotted] (3,-5) -- (3,-7);
  \draw[thick, dotted] (5,-5) -- (5,-7);
  \draw (4,-4) node {$\chi_2$};
  \draw (3,-7) node[below] {{\footnotesize $L$}};
  \draw (5,-7) node[below] {{\footnotesize $L+ \tfrac{2}{3} \pi$}};
  \draw (8,-7) -- (11,-5);
  \draw (11,-5) -- (13,-5);
  \draw (13,-5) -- (16,-7);
  \draw[thick, dotted] (11,-5) -- (11,-7);
  \draw[thick, dotted] (13,-5) -- (13,-7);
  \draw (12,-4) node {$\chi_3$};
  \draw (8.5,-7) node[below ] {{\footnotesize $2L+ \tfrac{2}{3} \pi$}};
  \draw (13,-7) node[below] {{\footnotesize $3L+ \tfrac{4}{3} \pi$}};
  \draw (16.5,-7) node[below] {{\footnotesize $4L+ \tfrac{4}{3} \pi$}};
\end{tikzpicture}
 \vspace{-0.5cm}
\end{center}
   \caption{the cut-off functions $\chi_i$ on $C_{\overline{k},\e}$}
\end{figure}

\vspace{-0.2cm}

 Then, we take the test $p$-forms $\vphi_i$ as follows:  \\[-0.6cm]
\begin{equation} \label{eq:test-p-form}
\begin{split}
 \vphi_i  &:=
 \begin{cases}
    \chi_i(r) \, v_{{\Bbb S}^p} 
       &  \text{on } H_1 = {\Bbb S}^{p} \times {\Bbb D}^{m-p}, \\
    \quad  0   
       &   \text{on } H_2 = \overline{M \setminus H_1},
 \end{cases}
\end{split}
\end{equation}
  where $v_{{\Bbb S}^p}$ is the volume form of ${\Bbb S}^p$.

 These $\vphi_i$ are co-closed $p$-form on $(M,g_{\e,p,\overline{k}}).$
 In fact, since the metric is product on $H_1$, we have on $H_1$
\begin{equation*} 
\begin{split}
  \delta_{g_{\e,p,\overline{k}}} \vphi_i 
  &= \delta_{g_{\e,p,\overline{k}}} \big( \chi_i \, v_{{\Bbb S}^p} \big) 
   = \big( \delta_{g_{\e,p,\overline{k}}} \chi_i \big) \, v_{{\Bbb S}^p}
      + \chi_i \, \big( \delta_{g_{\e,p,\overline{k}}} v_{{\Bbb S}^p} \big) \\
  &= 0 +  \chi_i \, \big( \delta_{g_{{\Bbb S}^p}} v_{{\Bbb S}^p} \big)
   \equiv 0.
\end{split}
\end{equation*}

\vspace{0.2cm}

 Since the supports of the family $\{ \vphi_i \}^{\overline{k}}_{i=1}$ are disjoint 
 up to measure $0$, the min-max principle for the Hodge-Laplacian on co-exact forms 
 gives us 
\begin{equation} \label{eq:co-exact-min-max}
\begin{split}
 \lambda^{\prime \prime (p)}_k (M, g_{\e,p,\overline{k}})
  &\le  \max_{i=1,2,\dots,\overline{k}} \left\{ 
     \dfrac{ \| d \vphi_i \|^2_{L^2 (M,g_{\e,p,\overline{k}})}  
      }{ \| \vphi_i \|^2_{L^2(M, g_{\e,p,\overline{k}})} } \right\}  \\[0.2cm]
  &= \max_{i=1,2,\dots,\overline{k}} \left\{ 
     \dfrac{ \| d \vphi_i \|^2_{L^2 (H_1, g_{\e,p,\overline{k}})}  
      }{ \| \vphi_i \|^2_{L^2(H_1, g_{\e,p,\overline{k}})} } \right\},
\end{split}
\end{equation}
 because of $\vphi_i \equiv 0$ on $H_2$.

 First, we estimate the numerator of \eqref{eq:co-exact-min-max} from above.
 Since 
\begin{equation} \label{eq:d-test-p-form}
\begin{split}
  d \vphi_i &= d \big( \chi_i(r) \, v_{{\Bbb S}^p}  \big)
   = \chi^{\prime}_i(r) dr \wedge v_{{\Bbb S}^p}  \quad  \text{ on } H_1
\end{split}
\end{equation}
 and $|\chi^{\prime}_1(r)|^2 = \dfrac{1}{L^2}$ on $[-L,0]$, 
 we have 
\begin{equation*}
\begin{split}
  \| d \vphi_1 \|^2_{ L^2 (H_1, g_{\e,p,\overline{k}}) }
   &= \dint_{ {\Bbb S}^p } \dint_{C_{\e,1}} 
       \big| \chi^{\prime}_1(r) dr \wedge v_{{\Bbb S}^p} \big|^2_{g_{\e,p,\overline{k}}} \,
          (\e \cosh(r))^{m-p-1} \, dr d \mu_{{\Bbb S}^{m-p-1}} d \mu_{{\Bbb S}^p} \\
   &= {\e}^{m-p-1} \, \dint^{0}_{-L} | \chi^{\prime}_1(r)|^2 \cosh^{m-p-1}(r) dr \,
       \dint_{ {\Bbb S}^{m-p-1} } d \mu_{{\Bbb S}^{m-p-1}} \,  
       \dint_{ {\Bbb S}^p } d \mu_{{\Bbb S}^p}  \\
   &= {\e}^{m-p-1} \, \dint^{0}_{-L} \dfrac{1}{L^2} \cosh^{m-p-1}(r) dr \cdot
      \vol ({\Bbb S}^{m-p-1}) \vol({\Bbb S}^{p})  \\
   &= \dfrac{{\e}^{m-p-1}}{L^2} \, \dint^{L}_{0} \cosh^{m-p-1}(r) dr \cdot
        \vol ({\Bbb S}^{m-p-1}) \vol({\Bbb S}^{p}) \\
   &\le  \dfrac{{\e}^{m-p-1}}{L^2} \cdot \dfrac{{\e}^{-(m-p-1)}}{m-p-1 } \cdot  
       \vol ({\Bbb S}^{m-p-1}) \vol({\Bbb S}^{p}) \quad 
        (\text{by } \eqref{eq:int-cosh^n})  \\
   &=  \dfrac{1}{(m-p-1) |\log \e|^2} \, \vol ({\Bbb S}^{m-p-1}) \vol({\Bbb S}^{p}),
\end{split}
\end{equation*}
 where $L = |\log \e|$ and $m-p-1 \ge 1$.

 Next, we estimate the denominator of \eqref{eq:co-exact-min-max} from below.
\begin{equation*}
\begin{split}
 \| \vphi_1 \|^2_{L^2(H_1, g_{\e,p,\overline{k}})} 
   &\ge  \dint_{ {\Bbb S}^p } \dint_{C_{\e,1}} 
      | \chi_1(r)|^2  \cdot \big| v_{{\Bbb S}^p} \big|^2_{g_{\e,p,\overline{k}}}
        \, d \mu_{g_{{\Bbb S}^p}} \, d \mu_{g_{C_{\e,1}}}  \\
   &\ge  \vol({\Bbb S}^{p}) \left\{ \dint_{B_1} d \mu_{g_{B_1}} + 
           \dint^{0}_{-L} |\chi_1(r)|^2 (\e \cosh(r))^{m-p-1} \, dr
           \vol({\Bbb S}^{m-p-1}) \right\} \\
   &\ge  \vol({\Bbb S}^{p}) \,
           \dint^{\tfrac{3}{4} \pi}_{\tfrac{\pi}{4}} \sin^{m-p-1}(r) dr 
           \vol({\Bbb S}^{m-p-1})  \\
   &\ge  \dfrac{1}{2^{(m-p-1)/2}} \cdot \dfrac{\pi}{2} \cdot
            \vol({\Bbb S}^{p}) \vol({\Bbb S}^{m-p-1}).
\end{split}
\end{equation*}

 Hence, for $\vphi_1$, we obtain an upper bound of the Rayleigh-Ritz quotient:
\begin{equation} \label{eq:Rayleigh-1}
\begin{split}
 \dfrac{ \| d \vphi_1 \|^2_{L^2 (H_1, g_{\e,p,\overline{k}})}  
      }{ \| \vphi_1 \|^2_{L^2(H_1, g_{\e,p,\overline{k}})} }
   &\le  \dfrac{C(m,p)}{|\log \e|^2},
\end{split}
\end{equation}
 where $C(m,p)$ is a positive constant depending only on $m,p$.

 In the same way, we can obtain similar upper bounds of the Rayleigh-Ritz 
 quotient for $\vphi_i$ ($i=2,3, \dots, \overline{k}$):
\begin{equation} \label{eq:Rayleigh-i}
\begin{split}
 \dfrac{ \| d \vphi_i \|^2_{L^2 (H_1, g_{\e,p,\overline{k}})}  
      }{ \| \vphi_i \|^2_{L^2(H_1, g_{\e,p,\overline{k}})} }
   &\le  \dfrac{C(m,p)}{|\log \e|^2},
\end{split}
\end{equation}
 where $C(m,p)$ is a positive constant depending only on $m,p$.

 Thus, we obtain 
\begin{equation} \label{eq:unnormalize-ev-estimate}
\begin{split}
 \lambda^{\prime \prime (p)}_k (M, g_{\e,p,\overline{k}})
   &\le \max_{i=1,\dots,\overline{k}} 
      \dfrac{ \| d \vphi_i \|^2_{L^2(M,g_{\e,p,\overline{k}})} }{
           \| \vphi_i \|^2_{L^2(M,g_{\e,p,\overline{k}})} } \\
   &\le \dfrac{C(m,p)}{| \log \e|^2} \longrightarrow 0
     \quad (\e \longrightarrow 0).
\end{split}
\end{equation}

 After the normalization of $g_{\e,p,\overline{k}}$, 
 from Lemma \ref{lem:basic} $(3)$, Lemma \ref{lem:M-vol-estimate}
 and \eqref{eq:unnormalize-ev-estimate}, it follows that 
\begin{equation*}
\begin{split}
 \lambda^{\prime \prime (p)}_k (M, \overline{g}_{\e,p,\overline{k}})
   &= \vol(M, g_{\e,p,\overline{k}})^{\frac{2}{m}} \cdot 
        \lambda^{(p)}_k (M, g_{\e,p,\overline{k}}) \\        
   &\le (A \overline{k} + B) ^{\frac{2}{m}} \cdot 
      \max_{i=1,\dots,\overline{k}} 
      \dfrac{ \| d \vphi_i \|^2_{L^2(M,\overline{g}_{\e,p,\overline{k}})} }{
           \| \vphi_i \|^2_{L^2(M,\overline{g}_{\e,p,\overline{k}})} } \\
   &\le  (A \overline{k} + B) ^{\frac{2}{m}} \, \dfrac{C(m,p)}{| \log \e|^2} 
     \longrightarrow 0  \quad (\e \longrightarrow 0),
\end{split}
\end{equation*}
 where $A, B >0$ are some constants independent of $\e$ and $\overline{k}$.

 This completes all the proofs.
\end{proof}

\vspace{0.2cm}

\begin{rem} \label{rem:small-ev}
 \renewcommand{\labelenumi}{$(\roman{enumi})$}
  From the proof of Lemma $\ref{lem:k-ev-estimate}$, we find the following:
\begin{enumerate}
 \item  One hyperbolic dumbbell yields one small eigenvalue.
  Thus, by gluing $\overline{k} = k+ b_p(M)$ hyperbolic dumbbells to embedded 
  spheres separately, we obtain another family of Riemannian metrics on $M$ 
  satisfying our desired properties.
 \item  For the rough Laplacian $\overline{\Delta} = \nabla^* \nabla$ acting on
  $p$-forms and tensor fields of any type, the same statement also holds.
  In fact, in the case of $p$-forms, for the test $p$-form in \eqref{eq:test-p-form},
  since $v_{{\Bbb S}^p}$ is parallel, 
  the same equality as in \eqref{eq:d-test-p-form} also holds:
\begin{equation*}
\begin{split}
  \nabla \vphi_i &= \nabla \big( \chi_i(r) \, v_{{\Bbb S}^p}  \big)
   = \chi^{\prime}_i(r) dr \otimes v_{{\Bbb S}^p}  \quad  \text{ on } H_1.
\end{split}
\end{equation*}
  In the case of $(a,b)$-tensor fields, 
  if we replace $v_{{\Bbb S}^p}$ by 
\begin{equation*}
\begin{split}
  \underbrace{ \dfrac{\partial}{\partial r} \otimes \dots \otimes 
    \dfrac{\partial}{\partial r} }_{\text{$a$-times}}
    \otimes \underbrace{dr \otimes \dots \otimes dr}_{\text{$b$-times}}
\end{split}
\end{equation*}
  for the test tensor fields like \eqref{eq:test-p-form}, 
  the same equality still holds.

 Therefore, by the same argument, there exist small $\overline{k} = k+ b_p(M)$ 
 eigenvalues of the rough Laplacian $\overline{\Delta}$ acting on $p$-forms 
 and tensor fields of any type on $(M, \overline{g}_{\e,p,\overline{k}})$.
\end{enumerate}
\end{rem}


 \section{Remarks and Further Studies}

 We discuss the future developments of the problem to find a positive lower bound 
 for the first positive eigenvalue of the Hodge-Laplacian on $p$-forms 
 in terms of geometric quantities (see also \cite{Honda-Mondino[24]}).
 It is well-known that the first positive eigenvalue of the Laplacian 
 on functions can be estimated below in terms of the dimension, 
 a lower bound of the Ricci curvature and an upper bound of the diameter 
 (\cite{Gromov[80]}, \cite{Li-Yau[80]}).
 In the case of $1 \le p \le m-1$, however, similar estimates do not hold any longer. 
 Typical counterexamples are the Berger spheres collapsing to the complex projective 
 spaces (see \cite{Colbois-Courtois[90]}). 
 In particular, their volumes converge to zero.

 Theorem \ref{thm:main} (or Theorem \ref{thm:uniform-main}) implies that 
 the first positive eigenvalue of the Hodge-Laplacian acting on $p$-forms cannot 
 be estimated below in terms of the dimension, the volume and a lower bound of 
 the sectional curvature. 
 From the proof, the diameter for the family of Riemannian metrics diverges 
 to infinity.
 It is a natural question to ask the case where the diameter is bounded in addition. 
 Colbois and Courtois \cite[Theorem 0.4]{Colbois-Courtois[90]} proved the following theorem:

\begin{thm}[Colbois and Courtois \cite{Colbois-Courtois[90]}] \label{thm:Colbois-Courtois[90]}
 For given $m \in \N$, $\kappa, v, D >0$, 
 there exists a positive constant $C(m, \kappa,v,D) >0$ 
 depending only on $m,\kappa,v$ and $D$ such that 
 any connected oriented closed Riemannian manifold $(M^m,g)$ of dimension $m$ with 
 $| K_g | \le \kappa$, $\vol(M,g) \ge v$ and $\diam (M,g) \le D$ 
 satisfies 
\begin{equation*} 
\begin{split}
  \lambda^{(p)}_1(M,g) &\ge C(m, \kappa,v,D) > 0
\end{split}
\end{equation*}
 for all $p=0,1,\dots,m$.
\end{thm}

 This theorem was proved by contradiction, by means of the $C^{1,\a}$-precompactness 
 theorem for $0< \a <1$ in the Lipschitz topology due to Peters \cite[Theorem 4.4]{Peters[87]}.
 So, this lower bound is implicit. 
 If the injectivity radius is bounded below away from zero in addition, 
 an explicit lower bound for $\lambda^{(p)}_1(M,g)$ is given by 
 \cite{Chanillo-Treves[97]}, \cite[Theorem 4.1]{Mantuano[08]}.

 If we weaken the curvature assumption of $| K_g | \le \kappa$ by $K_g \ge \kappa$, 
 then we do not know whether or not such a positive lower bound exists. 
 But, Lott \cite[p.918]{Lott[04]-quotient} conjectured the following:

\begin{conj}[Lott \cite{Lott[04]-quotient}] \label{conj:Lott}
 For given $m \in \N$, $\kappa \in \R$ and $v, D >0$,
 there would exist a positive constant $C(m, \kappa,v,D) >0$ 
 depending only on $m$, $\kappa$, $v$ and $D$ such that 
 any connected oriented closed Riemannian manifold $(M^m,g)$ of dimension $m$
 with $K_g \ge \kappa$, $\vol(M,g) \ge v$ and $\diam (M,g) \le D$ 
 satisfies 
\begin{equation*} 
\begin{split}
  \lambda^{(p)}_1(M,g) &\ge C(m, \kappa,v,D) > 0
\end{split}
\end{equation*}
 for all $p=0,1,\dots,m$.
\end{conj}

 This conjecture is still open, as far as the authors know.
 Recently, Honda and Mondino \cite{Honda-Mondino[25]} obtained 
 a positive lower bound of $\lambda^{(1)}_1(M,g)$ for $p=1$
 under $m \le 4$, the Ricci curvature $|\Ric_g | \le \kappa,$ $\vol(M,g) \ge v$ 
 and $\diam (M,g) \le D$.
 Their lower bound is also implicit with respect to $\kappa, v$ and $D$,
 since their proof is by contradiction, by means of the convergence theory 
 of Riemannian manifolds in dimension $4$ combined with 
 the convergence of the eigenvalues for $1$-forms \cite{Honda[17]}.
 Furthermore, they conjectured the following in the case of $p=1$:

\begin{conj}[Honda and Mondino \cite{Honda-Mondino[25]}] \label{conj:Honda-Mondino} 
 For given $m \in \N$, $\kappa \in \R$ and $v, D >0$, 
 there would exist a positive constant $C(m, \kappa,v,D) >0$ 
 depending only on $m,\kappa,v$ and $D$ such that 
 any connected oriented closed Riemannian manifold $(M^m,g)$ of dimension $m$ 
 with the Ricci curvature $\Ric_g \ge \kappa$, $\vol(M,g) \ge v$ and 
 $\diam (M,g) \le D$ satisfies 
\begin{equation*} 
\begin{split}
  \lambda^{(1)}_1(M,g) &\ge C(m, \kappa,v,D) > 0.
\end{split}
\end{equation*}
\end{conj}

\vspace{0.2cm}

 For given $m \in \N$, $\kappa \in \R$ and $v, D >0$, 
 we denote by $\mathcal{M}_K$ be the class of all connected oriented closed 
 Riemannian manifolds $(M^m,g)$ of dimension $m$ with 
 $K_g \ge \kappa$, $\vol(M,g) \ge v$ and $\diam (M,g) \le D$.
 For Conjecture \ref{conj:Lott}, we note that any sequence in $\mathcal{M}_K$ 
 is non-collapsing, and that $\mathcal{M}_K$ has only finite homeomorphism types 
 \cite{Grove-Petersen-Wu[90]}.
 Perelman \cite{Perelman[91]}, \cite{Kapovitch[07]} proved the topological 
 stability theorem for Alexandrov spaces:
 Let $X$ be a compact $m$-dimensional Alexandrov space with the (sectional) curvature 
 bounded below by $\kappa \in \R$. 
 Then, there exists a positive constant $\e = \e(X, \kappa) >0$ such that 
 if any compact $m$-dimensional Alexandrov space $Y$ with the curvature 
 bounded below by $\kappa$ satisfies $d_{GH}(X,Y)< \e$, 
 then $X$ is homeomorphic to $Y$.
 Here, $d_{GH}$ denotes the Gromov-Hausdorff distance. 

 Furthermore, Perelman claimed that the Lipschitz stability theorem held true. 
 Here, the Lipschitz stability theorem means that ``homeomorphic" 
 in the topological stability theorem above can be chosen to be ``bi-Lipschitz". 
 However, it seems that the paper has not appeared anywhere 
 (see Kapovitch \cite[p.104]{Kapovitch[07]}). 
 If the Lipschitz stability theorem would hold true, 
 then Conjecture \ref{conj:Lott} also holds true, 
 since the class $\mathcal{M}_K$ with the Lipschitz distance is covered with 
 finitely many balls.

\vspace{0.3cm}
 In the case of the Ricci curvature bounded below, instead of the sectional curvature,
 it would be considered that a similar lower bound for $2 \le p \le m-2$ 
 does not hold any longer.

\begin{prob} \label{prob:Ricci-non-collapsing}
 Do there exist a closed manifold $M$ of dimension $m$ and a sequence of 
 Riemannian metrics $g_i$ on $M$ with $\Ric_{g_i} \ge  (m-1) \kappa$, 
 $\vol(M,g_i) \ge v >0$ and $\diam (M,g_i) \le D$ for uniform constants 
 $\kappa \in \R$, $v, D >0$ independent of $g_i$ such that 
 for all $2 \le p \le m-2$
\begin{equation*} 
\begin{split}
  \lambda^{(p)}_1(M,g_i) &\longrightarrow 0  \quad (i \longrightarrow \infty) \  ?
\end{split}
\end{equation*}
\end{prob}

 We conjecture that the answer to this problem would be positive, however,
 there exist no such examples ever. 
 To find a lower bound for the first positive eigenvalue of the Hodge-Laplacian 
 acting on $p$-forms with $2 \le p \le m-2$, we may need to control 
 the Weitzenb\"ock curvature tensor. 
 But, it seems to be impossible to control the Weitzenb\"ock curvature tensor 
 for $2 \le p \le m-2$, in terms of  $\Ric_{g} \ge - (m-1) \kappa^2$, 
 $\vol (M,g) \ge v$ and $\diam (M,g) \le D$.

 Furthermore, the finiteness theorem for the Ricci curvature version fails. 
 For given $m \in \N$, $\kappa \in \R$ and $v, D >0$, we denote by 
 $\mathcal{M}_{\Ric}$ the class of all connected oriented closed Riemannian manifolds
 $(M^m,g)$ of dimension $m$ with $\Ric_g \ge  (m-1) \kappa$, $\vol(M,g) \ge v$ and 
 $\diam (M,g) \le D$.
 Then, due to the result by Perelman \cite{Perelman[97]},
 there exists infinitely many homeomorphism types in $\mathcal{M}_{\Ric}$.
 In particular, the $p$-th Betti number of a closed Riemannian manifold 
 in $\mathcal{M}_{\Ric}$, which is equal to the dimension of the harmonic $p$-forms,
 cannot be estimated above in terms of $m, \kappa, v$ and $D$.
 This situation is quite different from the case of the sectional curvature 
 bounded below
 (cf.\ the estimate of the total Betti number by Gromov \cite{Gromov[81]}).
 For further comments and remarks, see the comment by Lott in \cite[Remark $4.42$]{Lott[18]}.

\vspace{0.3cm}
 In contrast, Lott in \cite{Lott[18]} gave an upper bound of 
 $\lambda^{(p)}_k(M,g)$ under $K_g \ge \kappa$, $\diam(M,g) = D$ 
 and $\vol(M,g) \ge v$.
 More generally, including collapsing cases, he also gave an upper bound of 
 $\lambda^{(p)}_k(M,g)$ for $0 \le p \le n$, in terms of $K_g \ge \kappa$ and 
 the length $\ell >0$ of an $(n,1/10)$-strained point with $1 \le n \le m$, 
 instead of the assumption $\vol(M,g) \ge v$.

\vspace{0.2cm}



 \vspace{0.5cm}
\noindent
 \ Colette Ann\'e \\
 \quad  Laboratoire de Math\'ematiques Jean Leray, \\
 \quad  Nantes Universit\'e, CNRS, Facult\'e des Sciences, \\
 \quad  BP 92208, 44322 Nantes, France  \\
 \quad  colette.anne@univ-nantes.fr

\vspace{0.5cm}
\noindent
 \ Junya Takahashi \\
 \quad  Research Center for Pure and Applied Mathematics, \\
 \quad  Graduate School of Information Sciences, \\
 \quad  T\^{o}hoku University, \\
 \quad  $6$--$3$--$09$ Aoba, Sendai $980$--$8579$, Japan \\
 \quad  t-junya@tohoku.ac.jp

\end{document}